\documentclass[12pt]{article}

\usepackage{latexsym}
\usepackage{a4wide}
\usepackage{amscd}
\usepackage{graphics}
\usepackage{amsmath}
\usepackage{amssymb}
\usepackage{mathrsfs}
\usepackage{amsthm}
\usepackage{color}
\input xy
\xyoption{all}

\newcommand{\R}{\mathbb{R}}                     
\newcommand{\C}{\mathbb{C}}                     
\newcommand{\T}{\mathbb{T}}                     
\newcommand{\set}[2]{\left\{{#1}\mid{#2}\right\}}       
\newcommand{\ran}{\mathrm{ran\,}}		
\newcommand{\vol}{\mathrm{vol}}			

\newtheorem{thm}{\sc Theorem}      

\newtheorem{lem}{\sc Lemma}            

\title{Middle-dimensional squeezing and non-squeezing behavior of symplectomorphisms}

\author{Alberto Abbondandolo and Slava Matveyev}

\date{}

\begin{document}

\maketitle

\section*{Introduction}

Let $\Omega$ be the standard symplectic form
\[
\Omega = \sum_{j=1}^n dp_j \wedge dq_j
\]
on $\R^{2n}$, the standard Euclidean space endowed with coordinates $(q_1,p_1,\dots,q_n,p_n)$. The nonsqueezing theorem of Gromov states that no symplectic diffeomorphism (i.e.\ diffeomorphism which preserves $\Omega$) can map the $2n$-dimensional ball $B^{2n}(R)$ of radius $R$ into the cylinder
\[
\set{(q_1,p_1,\dots,q_n,p_n)\in \R^{2n}}{q_1^2+p_1^2 < S^2}
\]
if $S<R$ (see \cite{gro85}, and also \cite{eh89}, \cite{vit89}, \cite{hz94} for a different proof). This theorem shows that symplectic diffeomorphisms present two-dimensional rigidity phenomena (the base of the cylinder has dimension two), and not just the preservation of volume ensured by Liouville's theorem (which in the modern language is just a consequence that, preserving $\Omega$, a symplectic diffeomorphism must preserve also $\Omega^n$, the $n$-times wedge of $\Omega$ by itself, which is a multiple of the standard volume form). 

Since symplectic diffeomorphisms preserve also the $2k$-form $\Omega^k$ for every $1\leq k \leq n$, after Gromov's result it was natural to think that there should be also middle dimensional rigidity phenomena. A possible question concerned the possibility of symplectically embedding one polydisk $\Pi := B^2(R_1)\times \dots \times B^2(R_n)$ into another one $\Pi':=B^2(R_1')\times \dots \times B^2(R_n')$. If we adopt the standard convention that the radii of each of the two polydisks are increasing, Liouville's and Gromov's theorems immediately imply that if $\Pi$ can be symplectically embedded into $\Pi'$, then $R_1\cdots R_n \leq R_1' \cdots R_n'$ and $R_1\leq R_1'$. It was natural to expect other rigidity phenomena concerning other products of the $R_j$'s, but L.\ Guth recently ruled this out, by proving that there exists a constant $C(n)$ such that if $C(n)R_1\leq R_1'$ and $C(n) R_1 \cdots R_n \leq R_1' \cdots R_n'$, then $\Pi$ can be symplectically embedded into $\Pi'$ (see \cite{gut08}). See also \cite{sch05}, \cite{ms10}, \cite{hut10}, \cite{hk10} and references therein for more quantitative results about the symplectic embedding problem for polydisks and other domains.

In this article, we would like to take a different point of view and to keep the ball as the domain of our symplectic embeddings. We first notice that Gromov's nonsqueezing theorem can be restated by saying that every symplectic embedding $\phi:B^{2n}(R) \rightarrow \R^{2n}$ must satisfy the inequality
\begin{equation}
\label{area}
\mathrm{area} \bigl(P \phi (B^{2n}(R)) \bigr) \geq \pi R^2,
\end{equation}
where $P$ denotes the orthogonal projector onto the plane corresponding to the conjugate coordinates $q_1,p_1$. In fact, the latter statement is obviously stronger than the former. On the other hand, if the area of $A:= P \phi (B^{2n}(R))$ is smaller than $\pi R^2$, then we can find a smooth area preserving diffeomorphism $\psi:A\hookrightarrow B^2(S)$ for some $S<R$ (by a theorem of Moser \cite{mos65}, see also \cite[Introduction, Theorem 2]{hz94})), and the symplectic diffeomorphism $(\psi\times \mathrm{id}_{\R^{2n-2}}) \circ \phi$ maps $B^{2n}(R)$ into $Z(S)$, thus violating the former formulation of Gromov's nonsqueezing theorem. Actually, the reformulation (\ref{area}) is closer to Gromov's original proof.

In the above reformulation, the projector $P$ can be replaced by the orthogonal projector onto any complex line of $\R^{2n}\cong \C^n$, where the identification is given by $(q_1,p_1,\dots,q_n,p_n)\mapsto (q_1+ip_1,\dots,q_n+ip_n)$. If one does not wish to use the complex structure of $\R^{2n}$, (\ref{area}) can be reformulated using only the symplectic structure, by saying that if  $V$ is a symplectic plane in $\R^{2n}$ and $Q$ is the projector onto $V$ along the symplectic orthogonal complement of $V$, then
\begin{equation}
\label{area2}
\int_{Q \phi (B^{2n}(R))} \Omega \geq \pi R^2,
\end{equation}
for every symplectic embedding $\phi:B^{2n}(R) \rightarrow \R^{2n}$. In fact, (\ref{area2}) follows from (\ref{area}) because the projector $Q$ is conjugated to an orthogonal projector onto a complex line by a symplectic linear automorphism of $\R^{2n}$.

Looking at the inequality (\ref{area}), it seems natural to ask whether the nonsqueezing theorem has the following middle dimensional generalization: if $V$ is a complex linear subspace of $\R^{2n}$ of real dimension $2k$ and $P$ is the orthogonal projector onto $V$, is it true that
\begin{equation}
\label{middle}
\vol_{2k} \bigl(P \phi (B^{2n}(R)) \bigr) \geq \omega_{2k} R^{2k},
\end{equation}
for every symplectic embedding $\phi:B^{2n}(R) \rightarrow \R^{2n}$ ? Here $\omega_{2k}$ denotes the volume of the unit $2k$-dimensional ball. Indeed, the case $k=1$ is precisely Gromov's theorem, and for $k=n$ we have the equality in (\ref{middle}), by Liouville's theorem. The purely symplectic reformulation of this question, analogous to (\ref{area}), would be asking whether 
\begin{equation}
\label{middle2}
\frac{1}{k!}  \int_{Q \phi (B^{2n}(R))} \Omega^k  \geq \omega_{2k} R^{2k},
\end{equation}
when $Q$ is the projector onto a symplectic $2k$-dimensional linear subspace of $\R^{2n}$ along its symplectic orthogonal (the factor $k!$ appears because $\Omega^k$ restricts to $(k!)$-times the standard $2k$-volume form on every complex linear subspace of real dimension $2k$). 

The first aim of this note is to show that (\ref{middle}) (hence also (\ref{middle2})) holds in the linear category: 

\begin{thm}
Let $\Phi$ be a linear symplectic automorphism of $\R^{2n}$, and let $P: \R^{2n} \rightarrow \R^{2n}$ be the orthogonal projector onto a complex linear subspace $V\subset \R^{2n}$ of real dimension $2k$, $1\leq k\leq n$. Then
\[
\vol_{2k} \bigl( P \Phi (B^{2n}(R)) \bigr) \geq \omega_{2k} R^{2k},
\]
and the equality holds if and only if the linear subspace $\Phi^* V$ is complex.
\end{thm} 

The proof is elementary but, as in the case of the standard linear nonsqueezing (see \cite[Theorem 2.38]{ms98}), not completely straightforward. Our second aim is to show that in the nonlinear category this middle dimensional generalization of the nonsqueezing theorem is false:

\begin{thm}
\label{squeez}
Let $P:\R^{2n}\rightarrow \R^{2n}$ be the orthogonal projection onto a complex linear subspace of $\R^{2n}$ of real dimension $2k$, with $2\leq k \leq n-1$. For every $\epsilon>0$ there exists a smooth symplectic embedding $\phi: B^{2n} (1)\rightarrow \R^{2n}$ such that
\[
\mathrm{vol}_{2k} \bigl( P \phi(B^{2n} (1)) \bigr) < \epsilon.
\]
\end{thm}

The proof of this second result is based on some elementary but ingenious lemmata from the already mentioned paper of Guth \cite{gut08}.

Therefore, the middle-dimensional non-squeezing inequality (\ref{middle}) stops holding when passing from linear to nonlinear symplectic maps. However, the counterexample produced in the proof of Theorem 2 deforms the ball tremendously and it is natural to ask where the border of the validity of (\ref{middle}) lies. An interesting question seems to be: does (\ref{middle}) hold locally?

Such a question can be made precise in the following way. Given a symplectic diffeomorphism $\phi:\R^{2n} \rightarrow \R^{2n}$ and a point in $\R^{2n}$, without loss of generality the origin, is it true that (\ref{middle}) holds for every $R>0$ small enough? By rescaling, it is equivalent to ask whether the inequality (\ref{middle}) with $R=1$ holds when the symplectic embedding $\phi:B(1) \rightarrow \R^{2n}$ is $C^{\infty}$-close enough to a linear symplectic map $\Phi$. 

By the last assertion of Theorem 1, the answer is trivially positive when the subspace $\Phi^* V$ is not complex, just by continuity (in the former formulation, the hypothesis would be that $D\phi(0)^*V$ is not complex). But when $\Phi^* V$ is complex, for instance in the case of a $\phi$ which is $C^{\infty}$-close to the identity, we do not know the answer to this question. The last section of this article contains some remarks on such a problem and the suggestion that it might be related to an integrability issue.

\paragraph{Acknowledgements.}
This paper was influenced by stimulating discussions with Pietro Majer and Felix Schlenk. We wish to thank also Ivar Ekeland for some interesting suggestions about the local question, which we intend to develop in the future. 

\section{Linear non-squeezing}

We start by recalling the formulas for the volume of the image of the ball by a linear surjection. We denote by $B^n=B^n(1)$ the open unit ball about $0$ in $\R^n$. 

Let $n\geq k$ be positive integers and let $A:\R^n \rightarrow \R^k$ be linear and onto. We denote by $A^* : \R^k \rightarrow \R^n$ the adjoint of $A$ with respect to the Euclidean inner product. The linear mapping $A^* A:\R^n \rightarrow \R^n$  symmetric and semi-positive, with $k$-codimensional kernel 
\[
\ker A^* A = \ker A = (\ran A^*)^{\perp}.
\]
In particular, $A^*A$ restricts to an automorphism of the $k$-dimensional space $(\ker A)^{\perp} = \ran A^*$ and, since this restriction is the composition of the two isomorphisms
\[
A|_{(\ker A)^{\perp}} : (\ker A)^{\perp} \rightarrow \R^k, \quad A^* : \R^k \rightarrow (\ker A)^{\perp},
\]
which have the same determinant, being one the adjoint of the other, we deduce that
\[
\det \left( A^*A|_{(\ker A)^{\perp}} \right) = \Bigl| \det \left( A|_{(\ker A)^{\perp}} \right)\Bigr|^2.
\]
Here, the absolute value of the determinant of linear maps between different spaces of the same dimension is induced by the Euclidean inner products.
Let $\xi_1,\dots,\xi_k$ be a basis of $(\ker A)^{\perp}$ with
\[
|\xi_1 \wedge \dots \wedge \xi_k | = 1,
\]
where the Euclidean norm of $\R^n$ is extended to multi-vectors in the standard way (in particular, $|\xi_1 \wedge \dots \wedge \xi_k |$ is the $k$-volume of the prism generated by $\xi_1,\dots,\xi_k$). Since $A(B^n) = A(B^n\cap (\ker A)^{\perp})$, we find
\begin{equation}
\label{Volume}
\frac{\vol_k\bigl(A(B^n)\bigr)}{\omega_k}  = |A\xi_1 \wedge \cdots \wedge A\xi_k| =  |\det \left( A|_{(\ker A)^{\perp}} \right)| =
\sqrt{\det (A^*A|_{(\ker A)^{\perp}})},
\end{equation}
where $\omega_k$ denotes the $k$-volume of the unit $k$-ball.
Furthermore, the real function
\[
W \mapsto \bigl|\det A|_W\bigr| , \quad W \in \mathrm{Gr}_k(\R^n),
\]
where $\mathrm{Gr}_k(\R^n)$ denotes the Grassmannian of $k$-dimensional subspaces of $\R^n$, has a unique maximum at $(\ker A)^{\perp}= \ran A^*$, hence
\begin{equation}
\label{maxi}
\max_{W\in \mathrm{Gr}_k(\R^n)} \bigl|\det A|_W\bigr| = \bigl|\det A|_{\ran A^*} \bigr| .
\end{equation}

\medskip

Let $\R^{2n}$ be the $2n$-dimensional Euclidean space endowed with coordinates 
\[
(q_1,p_1,\dots,q_n,p_n),
\]
with the complex structure $i$ corresponding to the identification 
\[
(q_1,p_1,\dots, q_n,p_n) \equiv (q_1+ip_1, \dots,q_n+i p_n),
\]
and with the symplectic form given by the imaginary part of the corresponding Hermitian product, that is 
\[
\Omega = \sum_{j=1}^n dp_j \wedge dq_j.
\]
H.\ Federer refers to the next result as to the Wirtinger inequality, see \cite[section 1.8.1]{fed69}.

\begin{lem}
\label{comass}
Let $1\leq k\leq n$ and let $\Omega^k$ be the $k$-times wedge product of $\Omega$ by itself. Then 
\[
\bigl| \Omega^k [u_1,\dots,u_{2k} ] \bigr| \leq k! \, |u_1 \wedge \cdots \wedge u_{2k}|, \quad \forall u_1,\dots,u_{2k} \in \R^{2n},
\]
and, in the non-trivial case of linearly independent vectors $u_j$, the equality holds if and only if the $u_j$'s span a complex subspace.
\end{lem}

We are now ready to prove the linear non-squeezing result:

\setcounter{thm}{0}

\begin{thm}
Let $\Phi$ be a linear symplectic automorphism of $\R^{2n}$, and let $P: \R^{2n} \rightarrow \R^{2n}$ be the orthogonal projector onto a complex linear subspace $V\subset \R^{2n}$ of real dimension $2k$, $1\leq k\leq n$. Then
\[
\vol_{2k} \bigl( P \Phi (B^{2n}(R)) \bigr) \geq \omega_{2k} R^{2k},
\]
and the equality holds if and only if the linear subspace $\Phi^* V$ is complex.
\end{thm} 

\begin{proof}
By linearity, we may assume $R=1$. We consider the linear surjection
\[
A := P \Phi : \R^{2n} \rightarrow V.
\]
As before, let $\xi_1,\dots,\xi_{2k}$ be a basis of $(\ker A)^{\perp} = \ran A^* = \Phi^* V$ such that
\[
|\xi_1 \wedge \cdots \wedge \xi_{2k} | = 1.
\]
By the identity (\ref{Volume}) and Lemma \ref{comass},
\begin{equation}
\label{uno}
\begin{split}
\left( \frac{\vol_{2k} \bigl(P\Phi( B^{2n})\bigr)}{\omega_{2k}} \right)^2  = \det (A^* A|_{(\ker A)^{\perp}}) = |A^* A \xi_1 \wedge \cdots  \wedge A^* A\xi_{2k}| \\ \geq \frac{1}{k!} \bigl| \Omega^k [A^* A \xi_1,\dots,A^* A \xi_{2k} ] \bigr| = \frac{1}{k!} \bigl| \Omega^k [\Phi^* A \xi_1,\dots,\Phi^* A \xi_{2k} ] \bigr|,
\end{split} \end{equation}
and the equality holds if and only if the subspace spanned by $A^*A \xi_1, \dots, A^* A\xi_{2k}$, that is $\Phi^* V$, is complex.
Since $\Phi$ is symplectic, so is $\Phi^*$, hence
\begin{equation}
\label{due}
\Omega^k [\Phi^* A \xi_1,\dots,\Phi^* A \xi_{2k} ] = \Omega^k [A \xi_1,\dots,A \xi_{2k} ].
\end{equation}
Since the restriction of $|\Omega^k|$ to the complex subspace $V$ is $(k!)$-times the standard volume form, we have
\begin{equation}
\label{tre}
\frac{1}{k!} \bigl| \Omega^k [A \xi_1,\dots,A \xi_{2k} ] \bigr| = |A \xi_1 \wedge \cdots \wedge A\xi_{2k}| = \frac{\vol_{2k} \bigl(A(B^{2n})\bigr)}{\omega_{2k} } = \frac{\vol_{2k} \bigl(P \Phi (B^{2n})\bigr)}{\omega_{2k}},
\end{equation}
where we have used again (\ref{Volume}). By (\ref{uno}), (\ref{due}), and (\ref{tre}) we conclude that
\begin{equation}
\label{linnonsque}
\vol_{2k} (P \Phi (B^{2n})) \geq  \omega_{2k},
\end{equation}
and that the equality holds if and only if the linear subspace $\Phi^* V$ is complex.
\end{proof} 

\section{Nonlinear squeezing}

Let $B^{2n}(R)$ be the open ball about the origin of $\R^{2n}$ with radius $R$. As in the previous section, we just write $B^{2n}$ when the radius is 1.

We denote by $\Sigma$ the punctured torus $\T^2 \setminus \{\mbox{pt}\}$ equipped with a symplectic form of area 1. The following lemma is due to L.\ Guth \cite[Section 2, Main Lemma]{gut08}:

\begin{lem}
\label{guth1}
For every $R>0$ there exists a smooth symplectic embedding of $B^4(R)$ into $\Sigma \times \R^2$.
\end{lem}

The next lemma is a simple modification of Lemma 3.1 in \cite{gut08} (where an embedding with extra properties is constructed for $R<1/10$):

\begin{lem}
\label{guth2}
For every $R>0$ there exists a smooth symplectic embedding of $B^2(R)\times \Sigma$ into $\R^4$.
\end{lem}

\begin{proof}
Choose a positive number $\epsilon< R/3$ and set
\[
S := [-2R,2R] \times ]-\epsilon,\epsilon[, \quad S' := ]-\epsilon,\epsilon[ \times [-2R,2R].
\]
If $\epsilon$ is small enough (precisely, if $8R\epsilon<1$),  we can find a smooth symplectic immersion $\psi: \Sigma\rightarrow \R^2$ such that:
\begin{enumerate}
\item $\psi(\Sigma)\cap [-2R,2R]^2 = S \cup S'$;
\item $\psi^{-1}(S\cap S')$ consists of two disjoint open disks $D,D'\subset \Sigma$;
\item the restrictions $\psi|_{\Sigma \setminus \overline{D}}$ and $\psi|_{\Sigma \setminus \overline{D'}}$ are injective.
\end{enumerate}
Such a symplectic immersion is easily found by starting from a smooth immersion (see \cite[Figure 3]{gut08}) and by making it area-preserving by the already mentioned theorem of Moser). 

Let $\chi$ be a smooth real function on $\R$ with support in $[-2R+\epsilon,2R-\epsilon]$, such that $\chi=2R$ on $[-\epsilon,\epsilon]$ and $\|\chi'\|_{\infty} \leq 3/2$ (such a function exists because $\epsilon<R/3$). The map
\[
\phi: \R^4 \rightarrow \R^4, \quad \phi(q_1,p_1,q_2,p_2) := \bigl(q_1,p_1+\chi(q_2),
q_2,p_2+\chi'(q_2) q_1 \bigr),
\]
is a symplectic diffeomorphism, being the time-one map of the Hamiltonian flow generated by the Hamiltonian $H(q_1,p_1,q_2,p_2) := -\chi(q_2)q_1$. 

We claim that the map
\[
\varphi: B^2(R) \times \Sigma \rightarrow \R^4, \quad \varphi = \left\{ \begin{array}{ll} \mathrm{id} \times \psi & \mbox{on } B^2(R) \times (\Sigma\setminus \psi^{-1}(S)) , \\ \phi\circ (\mathrm{id} \times \psi) & \mbox{on } B^2(R) \times \psi^{-1}(S) , \end{array} \right.
\]
is a symplectic embedding. 

The map $\mathrm{id}\times \psi$ maps a neighborhood of the boundary of $B^2(R) \times \psi^{-1}(S)$ in $B^2(R) \times \Sigma$ into a small neighborhood of $B^2(R) \times \{\pm 2R\} \times ]-\epsilon,\epsilon[ $, on which $\phi=\mathrm{id}$. This proves that $\varphi$ is a smooth symplectic immersion. 

There remains to check that $\varphi$ is an embedding. By the properties of $\psi$, $\varphi$ is an embedding on a neighborhood of $B^2(R) \times \overline{\Sigma \setminus \psi^{-1}(S)}$ and on a neighborhood of $B^2(R) \times \psi^{-1}(S)$. Therefore, it is enough to prove that $\varphi$ maps the sets $B^2(R)\times (\Sigma \setminus \psi^{-1}(S))$ and $B^2(R)\times \psi^{-1}(S)$ into disjoint sets. Let $z\in \psi^{-1}(S)$, set $(q_2,p_2):= \psi(z)\in S$ and let $(q_1,p_1)\in B^2(R)$. Then
\begin{equation}
\label{imago}
\varphi(q_1,p_1,z) = (q_1, p_1 + \chi(q_2) , q_2, p_2 + \chi'(q_2)q_1).
\end{equation}
Since
\[
|p_2+\chi'(q_2) q_1| \leq \epsilon + \|\chi'\|_{\infty} R < \epsilon + \frac{3}{2} R < 2R,
\]
the point $\varphi(q_1,p_1,z)$ belongs to $\R^2 \times [-2R,2R]^2$. The intersection of the latter set with $B^2(R)) \times \varphi((\Sigma\setminus \psi^{-1}(S))$ is $B^2(R)\times S'$, so we must show that $\varphi(q_1,p_1,z)$ does not belong to $B^2(R)\times S'$. If $|q_2|\geq \epsilon$, (\ref{imago}) shows that the last two-dimensional component of $\varphi(q_1,p_1,z)$ does not belong to $S'$. If $|q_2|< \epsilon$, the second component of $\varphi(q_1,p_1,z)$ is 
\[
p_1 + 2R \geq R,
\]
hence the first two-dimensional component of $\varphi(q_1,p_1,z)$ does not belong to $B^2(R)$. This concludes the proof of Lemma \ref{guth2}.
\end{proof} 

We are now ready to prove the nonlinear squeezing result:

\begin{thm}
Let $P:\R^{2n}\rightarrow \R^{2n}$ be the orthogonal projection onto a complex linear subspace of $\R^{2n}$ of real dimension $2k$, with $2\leq k \leq n-1$. For every $\epsilon>0$ there exists a smooth symplectic embedding $\phi: B^{2n} \rightarrow \R^{2n}$ such that
\[
\mathrm{vol}_{2k} \bigl( P \phi(B^{2n}) \bigr) < \epsilon.
\]
\end{thm}

\begin{proof}
Up to the composition to a unitary automorphism of $(\R^{2n},i)$, we may assume that $V$ is the linear subspace corresponding to the coordinates $q_1,p_1,\dots,q_k,p_k$. Moreover, it is enough to consider the case $n=3$ and $k=2$, from which the general case follows by taking the product by the identity mapping. Because of these simplifications, $P:\R^6 \rightarrow \R^6$ is the standard projection on the subspace given by the first four coordinates $q_1,p_1,q_2,p_2$.

Let $R$ be a positive number. By Lemmata \ref{guth1} and \ref{guth2}, there are symplectic embeddings
\[
\varphi: B^4(R) \rightarrow \Sigma \times \R^2, \quad \psi: B^2(R) \times \Sigma \rightarrow \R^4.
\]
Consider the symplectic embedding $\tilde{\phi}:B^6(R) \rightarrow \R^6$ defined as the composition
\[
B^6(R) \hookrightarrow B^2(R) \times B^4(R) \stackrel{\mathrm{id}\times \varphi}{\longrightarrow} B^2(R) \times \Sigma \times \R^2 \stackrel{\psi \times \mathrm{id}}{\longrightarrow} \R^4 \times \R^2 = \R^6.
\]
Then
\begin{equation}
\label{volume}
\mathrm{vol}_4\bigl(P \tilde{\phi} (B^6(R)) \bigr) \leq \mathrm{vol}_4 \bigl(\psi(B^2(R) \times \Sigma)\bigr) = \mathrm{vol}_4 (B^2(R) \times \Sigma) = \pi R^2.
\end{equation}

The required symplectic embedding $\phi:B^6(1) \rightarrow \R^6$ is obtained by rescaling: Indeed, the embedding $\phi(z) := \tilde{\phi}(Rz)/R$ is symplectic and by (\ref{volume}), the quantity
\[
\mathrm{vol}_4 \bigl(P \phi(B^6(1)) \bigr) = \mathrm{vol}_4 \Bigl( \frac{1}{R} P \tilde{\phi}(B^6(R)) \Bigr) = \frac{1}{R^4} \mathrm{vol}_4\bigl(P \tilde{\phi} (B^6(R)) \bigr) \leq \frac{\pi}{R^2}
\]
is smaller than $\epsilon$, if $R$ is large enough.
\end{proof} 

\section{Remarks on the local question}

Let $\phi$ be a symplectic embedding of an open neighborhood $U$ of $0\in \R^{2n}$ into $\R^{2n}$ and let $P$ be the orthogonal projector onto $\R^{2k}$, the space spanned by $\partial/\partial q_1, \partial/\partial p_1, \dots, \partial/\partial q_k, \partial/\partial p_k$, with
$2\leq k \leq n-1$. Denote by $\psi$ the composition $P \phi$.
The local question proposed in the introduction is whether the inequality
\begin{equation}
\label{middle3}
\vol_{2k} \bigl(\psi (B^{2n}(R)) \bigr) \geq \omega_{2k} R^{2k},
\end{equation}
holds for $R>0$ small enough. 

The linear non-squeezing Theorem 1 implies that
\begin{equation}
\label{lincon}
J_{2k} \psi(x) \geq 1, \qquad \forall x\in U,
\end{equation}
where $J_{2k}\psi (x)$ is the $2k$-Jacobian of the map $\psi$ at $x$, that is the number
\[
J_{2k} \psi(x) = \max_{W\in \mathrm{Gr}_{2k}(\R^{2n})} |\det D\psi(x) |_W|.
\] 
Our first remark is that (\ref{middle3}) does not follow simply from the inequality (\ref{lincon}). In fact, if $m>h>1$, there are smooth maps $\varphi:\R^m \rightarrow \R^h$  whose $h$-Jacobian is everywhere at least 1 but for which
\[
\mathrm{vol}_h \bigl( \varphi(B^m(R)) \bigr) < \omega_h R^h
\]
for every small $R>0$. Examples with $m=3$ and $h=2$ can be found among the maps of the form
\[
\varphi: \C \times \R \cong \R^3 \rightarrow \C \cong \R^2, \quad
\varphi(z,t) = \rho(|z|) e^{it} z,
\]
where $\rho:[0,+\infty[ \rightarrow \R$ is a smooth positive function such that
\[
\rho(0)=1, \quad \rho'(0)=0, \quad \rho(r)<1 \quad \forall r>0.
\]
Indeed, such a $\varphi$ maps the cylinder $B^2(R)\times \R$ -- and a fortiori the ball $B^3(R)$ -- into the disk of radius $\rho(R)R$, which is smaller than $R$ for $R>0$. The $2$-Jacobian  of $\varphi$ is easily computed to be 
\[
J_2 \varphi (z,t)  = \rho(r) \bigl( \rho(r) + r \rho'(r) \bigr) \sqrt{1+r^2}, \quad \mbox{where } r=|z|,
\]
from which we find
\[
J_2 \varphi (z,t) = 1 + \left( \frac{1}{2} + 2\rho''(0) \right) |z|^2 + O(|z|^3) \quad \mbox{for } z\rightarrow 0.
\]
Therefore, if $\rho''(0)>-1/4$ then
$J_2 \varphi (z,t) \geq 1$ for $|z|$ small enough.

\medskip

By the linear non-squeezing Theorem 1, for every $x\in U$ the set
\[
\mathscr{W}(x) = \set{ W\in \mathrm{Gr}_{2k} (\R^{2n}) }{ |\det D\psi(x) |_W| \geq 1}
\]
is not empty. Our second remark is that if the ``multi-valued $2k$-dimensional distribution'' $\mathscr{W}$ is integrable, meaning that $U$ admits a $2k$-dimensional smooth foliation such that for every $x\in U$ the tangent space at $x$ of the leaf through $x$ belongs to $\mathscr{W}(x)$, then (\ref{middle3}) holds for every $R>0$ small enough. The proof of this fact is based on the following:

\begin{lem}
\label{foliation}
Let $\mathscr{F}$ be a $h$-dimensional foliation of $B^m(1)$, such that
each leaf is the graph of a smooth map from a strictly convex domain of $\R^h\times\{0\}$ to $\{0\}\times\R^{m-h}$.  Then $\mathscr{F}$ has a leaf $F$ such that the $h$-volume of $F$ is at least $\omega_h$.
\end{lem}

Let us sketch the proof of this lemma. Given $F$ in $\mathscr{F}$, let $\tilde{F}$ be a $h$-surface which minimizes the $h$-volume among all $h$-surfaces with boundary $\partial F = F\cap \partial B^m(1)$, which by F.\ Almgren's regularity theory is a minimal submanifold, which is smooth away from a singular set of codimension at least 2 (see \cite{alm00}). By the strict convexity assumption, $\tilde{F}$ is unique and depends continuously on $F\in \mathscr{F}$. In particular, there is a $F\in \mathscr{F}$ such that $0$ belongs to $\tilde{F}$. Then the conclusion follows from the monotonicity formula for minimal submanifolds (see e.g.\ \cite[Corollary 1.13]{cm11}), which yields
\[
\mathrm{vol}_h (F)  \geq
\mathrm{vol}_h (\tilde{F}) \geq \omega_h.
\]

\medskip

Now we prove that the existence of a smooth $2k$-dimensional foliation $\mathscr{F}$ of $U$ which is tangent to $\mathscr{W}$ (in the sense explained above) implies the local middle-dimensional non-squeezing statement. If $R$ is small enough, the homothety $x\mapsto x/R$ maps the foliation $\mathscr{F}|_{B^{2n}(R)}$ into a foliation of $B^{2n}(1)$ which satisfies the assumptions of Lemma \ref{foliation}. We deduce that $\mathscr{F}|_{B^{2n}(R)}$ has leaf $F$ such that
\[
\mathrm{vol}_{2k} (F) \geq \omega_{2k} R^{2k}.
\]
Since $D\psi(0)|_{T_0 F}$ is an isomorphism,
up to the choice of a smaller $R$ we can also assume that the restriction of $\psi$ to $F$ is injective. Then the fact that 
\[
\bigl|\det (D\psi (x)|_{T_x F}) \bigr|\geq 1
\]
and the area formula imply that
\[
\vol_{2k} \bigl(\psi (B^{2n}(R)) \bigr) \geq \vol_{2k} \bigl(\psi (F) \bigr) \geq \vol_{2k} (F) \geq 
\omega_{2k} R^{2k},
\]
as claimed.

\medskip

We conclude this article by discussing the issue of the integrability of $\mathscr{W}$, which as we have seen implies the local non-squeezing result. We do not know examples where such a condition fails. The problem in proving this condition is  not to find a smooth selection for $\mathscr{W}$ -- for instance the ``maximal expanding distribution''
\[
\widehat{W}(x) := D\phi(x)^* \R^{2k} 
\]
is such a smooth selection (see (\ref{maxi})) -- but to find an integrable one. The distribution  
$\widehat{W}$ defined above need not be integrable, even  in the case $k=1$. An example with $n=2$ and $k=1$ is the following. Set $(Q,P) := \phi(q,p)$, and notice that
\[
\widehat{W}(q,p) = D\phi(q,p)^* \R^2 =
\mathrm{span} \{ \nabla Q_1(q,p), \nabla P_1 (q,p) \}.
\]
Consider the generating function
\[
S(Q,p) := \frac{1}{2} p_2 Q_1^2,
\]
and let $\phi$ be the (local) symplectic diffeomorphism defined implicitly by
\[
Q =  q + \frac{\partial S}{\partial p} (Q,p), \quad
P = p - \frac{\partial S}{\partial Q} (Q,p).
\]
Then $Q_1 =q_1$, $P_1 = p_1 - p_2 q_1$, so
\[
\nabla Q_1 = \frac{\partial}{\partial q_1}, \quad \nabla P_1 = \frac{\partial}{\partial p_1} - p_2 \frac{\partial}{\partial q_1} - q_1 \frac{\partial }{\partial p_2}.
\]
The commutator of these two vector fields is
\[
[\nabla Q_1, \nabla P_1] = - \frac{\partial}{\partial p_2},
\]
which is nowhere in the subspace spanned by $\nabla Q_1$ and $\nabla P_1$. So $\widehat{W}$ is not integrable. However, $\mathscr{W}$ is integrable in this example: Indeed, the constant foliation parallel to the subspace $\R^2$ is tangent to $\mathscr{W}$.

Finally, we observe that in the ``rigid case'', namely when $\mathscr{W}(x)$ consists of a single vector space for every $x\in U$, the fact that $\phi$ is symplectic implies that $\mathscr{W}$ is integrable. In fact, in this case Theorem 1 implies that $\mathscr{W}(x)$ consists of the space
\[
\widehat{W}(x) = D\phi(x)^* \R^{2k},
\]
which must be complex for every $x\in U$.  Therefore, denoting by $J$ the complex structure of $\R^{2n}$, the distribution
\[
\widehat{W}(x) = J\widehat{W}(x) = JD\phi(x)^*\R^{2k} = D\phi(x)^{-1}J\R^{2k} = D\phi(x)^{-1}\R^{2k}
\]
is tangent to the foliation given by the image by $\phi^{-1}$ of the linear foliation given by $2k$-dimensional subspaces parallel to $\R^{2k}$. Summarizing, the case in which we cannot prove the inequality (\ref{middle3}) for small $R$ is the following: $1<k<n$, $J_{2k}\psi(0)=1$, but in every neighborhood of $0$ there are points $x$ such that $J_{2k}\psi(x)>1$.

\providecommand{\bysame}{\leavevmode\hbox to3em{\hrulefill}\thinspace}
\providecommand{\MR}{\relax\ifhmode\unskip\space\fi MR }
\providecommand{\MRhref}[2]{%
  \href{http://www.ams.org/mathscinet-getitem?mr=#1}{#2}
}
\providecommand{\href}[2]{#2}


\begin{thebibliography}{Alm00}

\bibitem[Alm00]{alm00}
F.~J. Almgren, \emph{Almgren's big regularity paper}, World Scientific
  Monograph Series in Mathematics, vol.~1, World Scientific Publishing, River
  Edge, NJ, 2000, Edited by J. E. Taylor and V.\ Scheffer.

\bibitem[CM11]{cm11}
T.~H. Colding and W.~P.~Minicozzi II, \emph{A course in minimal surfaces},
  Graduate Studies in Mathematics, American Mathematical Society, 2011.

\bibitem[EH89]{eh89}
I.~Ekeland and H.~Hofer, \emph{Symplectic topology and {H}amiltonian dynamics},
  Math. Z. \textbf{200} (1989), 355--378.

\bibitem[Fed69]{fed69}
H.~Federer, \emph{Geometric measure theory}, Springer, 1969.

\bibitem[Gro85]{gro85}
M.~Gromov, \emph{Pseudo holomorphic curves in symplectic manifolds}, Invent.
  Math. \textbf{82} (1985), 307--347.

\bibitem[Gut08]{gut08}
L.~Guth, \emph{Symplectic embeddings of polydisks}, Invent. Math. \textbf{172}
  (2008), 477--489.

\bibitem[HK10]{hk10}
R.~Hind and E.~Kerman, \emph{New obstructions to symplectic embeddings}, {\tt
  arXiv :0906.4296v2 [math.SG]}, 2010.

\bibitem[Hut10]{hut10}
M.~Hutchings, \emph{The embedded contact homology and its applications},
  Proceedings of the {I}nternational {C}ongress of {M}athematicians (Hyderabad,
  India), 2010.

\bibitem[HZ94]{hz94}
H.~Hofer and E.~Zehnder, \emph{Symplectic invariants and {H}amiltonian
  dynamics}, Birkh\"auser, Basel, 1994.

\bibitem[Mos65]{mos65}
J.~Moser, \emph{On the volume elements of a manifold}, Trans. Amer. Math. Soc.
  \textbf{120} (1965), 286--294.

\bibitem[MS98]{ms98}
D.~McDuff and D.~Salamon, \emph{Introduction to symplectic topology}, second
  ed., Oxford Mathematical Monographs, The Clarendon Press Oxford University
  Press, New York, 1998.

\bibitem[MS10]{ms10}
D.~McDuff and F.~Schlenk, \emph{The embedding capacity of 4-dimensional
  symplectic ellipsoids}, Ann. of Math. (to appear), {\tt arXiv:0912.0532v2
  [math.SG]}, 2010.

\bibitem[Sch05]{sch05}
F.~Schlenk, \emph{Embedding problems in symplectic geometry}, de Gruyter
  Expositions in Mathematics, vol.~40, Walter de Gruyter, 2005.

\bibitem[Vit89]{vit89}
C.~Viterbo, \emph{Capacit\'e symplectiques et applications}, Ast\'erisque
  \textbf{177-178} (1989), no.~714, S\'eminaire Bourbaki 41\'eme ann\'ee,
  345--362.

\end{thebibliography}
\end{document}